\documentclass[12pt]{amsart}
\usepackage[centertags]{amsmath}
\usepackage{latexsym}
\usepackage{amssymb}

\DeclareMathOperator{\Disc}{Disc}

\DeclareMathOperator{\TSRI}{\mathrm{TSRI}}

\newcommand{\F}{\mathbb F}
\newcommand{\barF}{\overline{\mathbb F}}

\newcommand{\GL}{\mathrm{GL}}

\newcommand{\hide}[1]{}
\newtheorem{dummy}{Dummy}

\newtheorem{lemma}[dummy]{Lemma}

\newtheorem{theorem}[dummy]{Theorem}

\theoremstyle{definition}

\theoremstyle{remark}

\newtheorem{rem}[dummy]{Remark}

\begin{document}

\bibliographystyle{amsalpha}
\author{Sandro Mattarei}
\email{smattarei@lincoln.ac.uk}
\address{Charlotte Scott Centre for Algebra\\
University of Lincoln \\
Brayford Pool
Lincoln, LN6 7TS\\
United Kingdom}
\author{Marco Pizzato}
\email{marco.pizzato1@gmail.com}
\address{Dipartimento di Matematica\\
  Universit\`a degli Studi di Trento\\
  via Sommarive 14\\
  I-38123 Povo (Trento)\\
  Italy}
\title{Irreducible polynomials from a cubic transformation}
\begin{abstract}
Let $R(x)=g(x)/h(x)$ be a rational expression of degree three over the finite field $\F_q$.
We count the irreducible polynomials in $\F_q[x]$, of a given degree, which have the form
$h(x)^{\deg f}\cdot f\bigl(R(x)\bigr)$ for some $f(x)\in\F_q[x]$.
As an application, we recover the number of irreducible transformation shift registers of order three,
previously computed by Jiang and Yang.
\end{abstract}

\maketitle

\section{Introduction}\label{sec:intro}

Let $R(x)$ be a rational expressions of degree $r$ over a finite field $\F_q$.
Thus, $R(x)$ has the form $R(x) = g(x)/h(x)$, where $g(x)$ and $h(x)$ are coprime polynomials in $\F_q[x]$ with $\max\{\deg g,\deg h\} = r$.
To $R(x)$ (or, to be rigorous, to the pair of polynomials $g(x)$, $h(x)$, which is determined by $R(x)$ only up to scalar multiples)
we associate a {\em transformation} of polynomials in $\F_q[x]$,
given by sending a polynomial $f(x)$ in $\F_q[x]$ to the polynomial
$f_R(x)=h(x)^n\cdot f\bigl(g(x)/h(x)\bigr)$.

The case $r=1$ is of little interest, but {\em quadratic transformations} (that is, with $r=2$) occur naturally.
As a notable example, any {\em self-reciprocal} polynomial $F(x)\in \F_q[x]$ (which means satisfying $x^{\deg F}\cdot F(1/x)=F(x)$)
which does not have $1$ or $-1$ as a root can be written in the form
$F(x)=x^{\deg f}\cdot f(x+1/x)$ for some polynomial $f(x)\in\F_q[x]$,
hence $F(x)=f_R(x)$  with $R(x)=(x^2+1)/x$.

The number of self-reciprocal irreducible monic polynomials of a given degree over $\F_q$ was determined by Carlitz in~\cite{Carlitz:srim},
and later investigated by various other authors.
In particular, in~\cite{Ahmadi:Carlitz} Ahmadi extended Carlitz's formula to a count of irreducible polynomials
obtained from an arbitrary quadratic transformation.

The fact that the counting formula found by Ahmadi is essentially the same as Carlitz's formula for the special case $R(x)=(x^2+1)/x$,
at least for $q$ odd, was given a simple explanation in~\cite{MatPiz:self-reciprocal}.
It depends on the fact that such enumeration is essentially unaffected
when the quadratic rational expression $g(x)/h(x)$ is replaced with any other obtained from it by composition, on both sides,
with independent M\"obius transformations (meaning rational expressions of degree $1$).
We call such rational expressions {\em equivalent} in this paper.
It is not hard to see that all quadratic rational expressions $g(x)/h(x)$ over an algebraically closed field of odd
(or zero) characteristic are equivalent.
Over a finite field $\F_q$ of odd characteristic they split into two equivalence classes.
However, the rational expressions in either class happen to give the same enumeration formula for the irreducible polynomials $F(x)$
obtained from the corresponding transformations.
Incidentally, because in odd characteristic the quadratic rational expression $(x^2+1)/x$ is equivalent to $x^2$,
that enumeration formula also counts the number of irreducible polynomial of the form $f(x^2)$ over the field $\F_q$ with $q$ odd,
which is a very special case of a result of Cohen~\cite[Theorem~3]{Cohen:irreducible}.

These results motivate an investigation of irreducible polynomials obtained from a transformation of higher degree,
that is, of the form $f_R(x)$ where $R(x)=g(x)/h(x)$ a rational expression of arbitrary degree $r$.
The special case where $g(x)=1$ is of particular relevance for practical applications
because of its direct connection with transformation shift registers (TSRs), see~\cite[Section~3]{Cohen+:TSR}.

In this paper we study the case of {\em cubic transformations}.
Given coprime polynomials $g(x),h(x)\in\F_q[x]$,
following notation in~\cite{Cohen+:TSR} we let $\mathcal{I}(g,h,n,q)$
denote the set of monic irreducible polynomials $f(x)\in\F_q[x]$ of degree $n>1$
such that $f_R(x)$ is irreducible in $\F_q[x]$.
In Section~\ref{sec:tools} we extend the method used in~\cite{Ahmadi:Carlitz} for quadratic transformations,
to a general procedure for computing the cardinality of $\mathcal{I}(g,h,n,q)$
when $R(x)=g(x)/h(x)$ is an arbitrary cubic rational expression over $\F_q$.

Because the procedure of Section~\ref{sec:tools} involves the number of $\F_q$-rational points of a certain
cubic or quartic affine plane curve over $\F_q$ associated with the rational expression $R(x)$,
it is not possible to obtain a general explicit formula for $|\mathcal{I}(g,h,n,q)|$, covering all cases.
This is in contrast with the case of quadratic transformations as studied in~\cite{Ahmadi:Carlitz,MatPiz:self-reciprocal}.
However, we find explicit formulas in all cases where that is possible, which include
those arising from TSRs.

Because $|\mathcal{I}(g,h,n,q)|$ is unaffected by replacing $R(x)=g(x)/h(x)$ with an equivalent rational expression,
as we show in Section~\ref{sec:equivalence}, it is useful to have a classification, or at least a partial classification,
up to equivalence, of cubic rational maps over a finite field.
Such a classification is produced in the paper~\cite{MatPiz:cubic-maps},
at least for cubic rational expressions having at most three ramification points (over the algebraic closure $\barF_q$).
We quote it here as Theorem~\ref{thm:cubic_expr_finite} for characteristic at least five, where it provides nice representatives
for the resulting four equivalence classes of cubic rational expressions.
In Section~\ref{sec:three_ram_points} we find explicit formulas for $|\mathcal{I}(g,h,n,q)|$
in case $R(x)=g(x)/h(x)$ equals one of those four representatives.
When $R(x)$ has four ramification points the cubic or quartic curve mentioned above is irreducible of genus one,
so we must content ourselves with an estimate, coming from the Hasse-Weil bound,
which we state in Theorem~\ref{thm:four_ram_points}.

A classification of cubic rational expressions up to equivalence is different in characteristic three or two,
and we quote the corresponding results from~\cite{MatPiz:cubic-maps}
as Theorems~\ref{thm:cubic_expr_finite_3} and~\ref{thm:cubic_expr_finite_2}.
In Sections~\ref{sec:char_three} and~\ref{sec:char_two} we compute $|\mathcal{I}(g,h,n,q)|$ in those cases.
Note that in the case of characteristic two, which is likely the most important for applications,
no cubic rational expression can have more than three ramification points, hence Theorem~\ref{thm:cubic_expr_finite_2}
provides a classification of all cubic rational expressions.
Still, a curve of genus one arises in some of the equivalence classes,
leading again to an estimate in those cases rather than an explicit formula.

In the final Section~\ref{sec:TSR} we apply our results to obtain an alternate derivation of the formulas for
the number of irreducible TSRs of order three over $\F_{q^m}$, originally obtained by Jiang and Yang in~\cite{Jiang-Yang}.

The research leading to this paper began when the second author was a PhD student at the University of Trento, Italy,
under the supervision of the first author.
Part of these results have appeared among other results in~\cite{Pizzato:thesis}.

\section{Rational expressions and associated polynomial transformations}\label{sec:equivalence}

Consider a rational expression $R(x)=g(x)/h(x)\in K(x)$, over a field $K$, where $g(x)$ are $h(x)$ coprime polynomials in $K[x]$.
Its {\em degree} $\deg R$ is $\max(\deg g,\deg h)$;
the expression will be called {\em linear, quadratic, cubic,}
when its degree equals $1,2,3$.
It is easy to see that the degree of the composite of two rational expressions
equals the product of their degrees.
The linear rational expressions
(which also go by various other names, such as {\em fractional linear transformations,} or {\em M\"obius transformations}) form a group,
usually called {\em the M\"obius group} over $K$.

To a rational expression $R(x)$ over the field $K$
we associate a {\em transformation} of polynomials in $K[x]$,
which sends (zero to zero if we like, and) a nonzero polynomial $f(x)$ to the polynomial
$f_R(x):=h(x)^{\deg f}\cdot f\bigl(g(x)/h(x)\bigr)$.
Thus, the transformation is given by the substitution $x\mapsto R(x)$ into $f(x)$,
followed with multiplication
by the least power of $h(x)$ required to clear denominators and ensure that $f_R(x)$ is actually a polynomial.
Note that $f_R(x)$ changes by a nonzero scalar factor if we rewrite
$R(x)=g(x)/h(x)$ in an equivalent form $\bigl(ag(x)\bigr)/\bigl(ah(x)\bigr)$.
Hence, strictly speaking, $f_R(x)$ depends on $g(x)$ and $h(x)$ rather than just their quotient $R(x)$,
but we will abuse notation and assume that we have fixed a representation of $R(x)$ as $g(x)/h(x)$
when using the notation $f_R(x)$.

For a rational expression $R(x)$ of arbitrary degree $r$ we clearly have $\deg f_R\leq r\deg f$,
but equality need not always hold.

\begin{lemma}\label{lemma:degree_drop}
Let $R(x)=g(x)/h(x)$ be a rational expression of degree $r\ge 1$ over $K$,
and write $g(x)=\sum_i g_ix^i$ and $h(x)=\sum_i h_ix^i$.
Then
\begin{itemize}
\item[(a)]
$\deg f_R=r\deg f$ unless $h_r\neq 0$ and $f(g_r/h_r)=0$;
\item[(b)]
$\deg f_R=0$ if and only if $h_rg(x)-g_rh(x)$ is a constant and
$f(x)$ is a scalar multiple of a power of $h_rx-g_r$.
\end{itemize}
\end{lemma}

\begin{proof}
Writing $f(x)=\sum_if_ix^i$ and $n=\deg f$, we see that the polynomial
$f_R(x)=\sum_if_i\,g(x)^i\,h(x)^{n-i}$
has degree at most $rn$, and the coefficient of $x^{rn}$ in it equals
$\sum_if_i\,g_r^i\,h_r^{n-i}=h_r^n\,f(g_r/h_r)$.
Assertion~(a) follows.

Regarding Assertion~(b), if $f(x)=(h_rx-g_r)^n$ then
$f_R(x)
=\bigl(h_rg(x)-g_rh(x)\bigl)^n$,
and this is a (nonzero) constant if $h_rg(x)-g_rh(x)$ is.
To prove the converse implication,
we may assume $h_r\neq 0$ because of Assertion~(a).
It is also enough to prove the statement with $R(x)$ replaced by
$\tilde R(x)=\tilde g(x)/h(x)=R(x)-g_r/h_r$,
whence $h_r\tilde g(x)=h_rg(x)-g_rh(x)$
has degree less than the degree of $h(x)$.
Thus, suppose $\deg f_{\tilde R}=0$, that is, $f_{\tilde R}(x)=f_R(x-g_r/h_r)$ is a nonzero constant.
Writing $\tilde g(x)=\sum_i\tilde g_ix^i$,
if $x^m$ is the highest power of $x$ which divides $f(x)$, then the term $f_m\,\tilde g(x)^m\,h(x)^{n-m}$
has higher degree than each other term in the sum
$f_{\tilde R}(x)=\sum_if_i\,\tilde g(x)^i\,h(x)^{n-i}$.
This forces $m=n$, and then $f_n\tilde g(x)^n=1$ implies that $\tilde g(x)$ is a constant, as desired.
\end{proof}

With natural projective interpretation,
Assertion~(a) of Lemma~\ref{lemma:degree_drop} says that the drop in degree $\deg f_R<r\deg f$
occurs when the value of $R(x)$ at $\infty$ is a root of $f(x)$.
In particular, this cannot occur when $f$ is irreducible and $\deg f>1$.
Furthermore,  because in general $f_R(x)=u_R(x)\,v_R(x)$ if $f(x)=u(x)\,v(x)$,
Assertion~(b) of Lemma~\ref{lemma:degree_drop} shows that
$f_R(x)$ cannot be irreducible unless $f(x)$ is irreducible or,
only in case $h_rg(x)-g_rh(x)$ is a constant,
$f(x)$ is the product of an irreducible polynomial $u(x)$ and a power of $h_rx-g_r$.
Because in the latter case we have $f_R(x)=u_R(x)$,
it is harmless to assume $f(x)$ to be irreducible
when studying irreducible polynomials of the form $f_R(x)$.

Following notation in~\cite{Cohen+:TSR} we let $\mathcal{I}(g,h,n,q)$ denote the set of monic irreducible polynomials $f(x)\in\F_q[x]$ of degree $n>1$
such that $f_R(x)$ is irreducible in $\F_q[x]$.
The possibility of $n=1$ is deliberately excluded from the notation because this case in uninteresting and
its inclusion would create tedious case distinctions.
The following result shows that our goal of counting irreducible polynomials of the form $f_R(x)$
is not affected by composing
$g(x)/h(x)$ with M\"obius transformations on either side.

\begin{lemma}\label{lemma:invariance}
Let $R(x)=g(x)/h(x)$ be a rational expression of degree $r\ge 1$ over $\F_q$.
Let $A(x)$ and $B(x)$ be linear rational expressions over $\F_q$, and write
$(B\circ R\circ A)(x)=\tilde R(x)=\tilde g(x)/\tilde h(x)$.
Then for $n>1$ we have
\[
|\mathcal{I}(\tilde g,\tilde h,n,q)|=|\mathcal{I}(g,h,n,q)|.
\]
\end{lemma}

\begin{proof}
Because of our restriction $n>1$, according to Lemma~\ref{lemma:degree_drop} we have
$\deg f_R=r\deg f=rn$ for every irreducible polynomial $f(x)\in K[x]$ of degree $n$.

Since the M\"obius group is generated (under composition) by
the expressions $ax+b$, with $a\in \F_q^\ast$ and $b\in \F_q$, and $x\mapsto 1/x$,
in order to prove Lemma~\ref{lemma:invariance} it suffices to show that its conclusion holds in the special cases where one of $A(x)$ and $B(x)$
is one of those special expressions, and the other is $x$ (the identity of the M\"obius group).

When $A(x)=ax+b$ and $B(x)=x$, the polynomial $f_{\tilde R}(x)=f_R(ax+b)$ is clearly irreducible if and only if $f_R(x)$ is.
Hence in this case we have
$\mathcal{I}(\tilde g,\tilde h,n,q)=\mathcal{I}(g,h,n,q)$.

When $A(x)=1/x$ and $B(x)=x$, up to scalar factors we may take $\tilde g(x)=x^rg(1/x)$ and $\tilde h(x)=x^rh(1/x)$, whence
\[
f_{\tilde R}(x)=
\bigl(x^rh(1/x)\bigr)^{\deg f}\cdot f\bigl(g(1/x)/h(1/x)\bigr).
\]
Because $\deg f_R=r\deg f$ this equals the {\em reciprocal} polynomial $x^{\deg f_R}f_R(1/x)$ of $f_R(x)$,
which is irreducible if and only $f_R(x)$ is.
Once again we find
$\mathcal{I}(\tilde g,\tilde h,n,q)=\mathcal{I}(g,h,n,q)$.

If $B(x)$ does not equal a scalar multiple of $x$, the sets
$\mathcal{I}(\tilde g,\tilde h,n,q)$ and $\mathcal{I}(g,h,n,q)$ are generally different,
because the irreducibility of $f_{\tilde R}(x)$ does not directly relate to the irreducibility of $f_R(x)$.

However, when $A(x)=x$ and $B(x)=ax+b$,
whence $\tilde R(x)=ag(x)/h(x)+b$,
the map $f(x)\mapsto\tilde f(x)=f(ax+b)$
is a degree-preserving bijection from the set of irreducible polynomials $f(x)$ such that
$f_{\tilde R}(x)$ is irreducible, onto the set of irreducible polynomials $\tilde f$ such that
$\tilde f_R(x)$ is irreducible, because
$f_{\tilde R}(x)=\tilde f_R(x)$.
In particular, the map $f(x)\mapsto\tilde f(x)=a^{-\deg f}f(ax+b)$
gives a bijection from $\mathcal{I}(\tilde g,\tilde h,n,q)$ onto $\mathcal{I}(g,h,n,q)$.

Similarly, when $A(x)=x$ and $B(x)=1/x$, whence $\tilde R(x)=h(x)/g(x)$,
the map $f(x)\mapsto\tilde f(x)=x^{\deg f}f(1/x)$
is a degree-preserving bijection from the set of irreducible polynomials $f$ with $\deg f>1$ such that
$f_{\tilde R}(x)$ is irreducible, onto the set of irreducible polynomials $\tilde f$ with $\deg\tilde f>1$ such that
$\tilde f_R(x)$ is irreducible,
again because
$f_{\tilde R}(x)=\tilde f_R(x)$.
Here the assumption $\deg f>1$ serves to exclude the exceptional case $f(x)=x$.
As in the previous case it follows that
$|\mathcal{I}(\tilde g,\tilde h,n,q)|=|\mathcal{I}(g,h,n,q)|$,
and this concludes the proof.
\end{proof}

According to Lemma~\ref{lemma:invariance}, for the purpose of counting the irreducible polynomials
of the form $f_R(x)$ and a given degree over a finite field $K=\F_q$,
we may take advantage of any normalization which replaces $R(x)$ with some
$\tilde R(x)=(B\circ R\circ A)(x)$
having a simpler form.
Call two such $R(x)$ and $\tilde R(x)$ {\em equivalent}.

In~\cite{MatPiz:cubic-maps} we studied the 
equivalence classes of cubic rational expressions $R(x)=g(x)/h(x)$ over a finite field,
and gave representatives for those equivalence classes of expressions having at most three ramification points (over the algebraic closure $\barF_q$).
For the purposes of this paper, the {\em ramification points} of a rational expression $R(x)$ are the zeroes of its derivative,
which are the same as the zeroes of the polynomial $g'(x)h(x)-g(x)h'(x)$,
together with $\infty$ in case this polynomial has degree less than $2\deg R-2$.
In particular, a cubic rational expression $R(x)$ has at most four ramification points.

\begin{theorem}[\cite{MatPiz:cubic-maps}, Theorem~14]\label{thm:cubic_expr_finite}
Let $\F_q$ be a finite field of characteristic at least five,
and let $\sigma\in\F_q$ be a nonsquare.
Then any cubic rational expression $R(x)$ over $\F_q$ with at most three ramification points (over $\barF_q$) is equivalent to either
$x^3$, or $(x^3+3\sigma x)/(3x^2+\sigma)$, or $x^3-3x$, or $x^3-3\sigma x$.
\end{theorem}

The expressions $x^3$ and $(x^3+3\sigma x)/(3x^2+\sigma)$ have two ramification points ($0$ and $\infty$ in case of the former)
and, in fact, are equivalent over $\F_{q^2}$ (see the proof of~\cite[Theorem~14]{MatPiz:cubic-maps}).
The expressions $x^3-3x$ and $x^3-3\sigma x$ have three ramification points.
The difficulties of extending the above classification up to equivalence to cubic rational expressions with
four ramification points are discussed in~\cite[Remark~15]{MatPiz:cubic-maps}.

\section{A general procedure}\label{sec:tools}
In this section we describe a method to compute $|\mathcal{I}(g,h,n,q)|$
for a given cubic transformation $R(x)=g(x)/h(x)$.
We will use the following standard fact, known as {\em Capelli's lemma,} see~\cite[Lemma~1]{Cohen:irreducible} for a proof.

\begin{lemma}[Capelli]\label{Capelli}
Let $f(x)\in\F_q[x]$ be an irreducible polynomial of degree $n$ and let $g,h \in\F_q[x]$.
Then $h(x)^n\cdot f\bigl(g(x)/h(x)\bigr)$ is irreducible in $\F_q[x]$
if and only if $g(x)-\beta h(x)$ is irreducible in
$\F_{q^n}[x]$, where $\beta\in\F_{q^n}$ is any root of $P(x)$.
\end{lemma}

In what follows we assume, as we may after post-composing with a suitable M\"obius transformation, that $h(x)$ has degree less than three,
and that $g(x)$ is monic.
The main advantage of this assumption is that $g(x)-\beta h(x)$ is then always a monic cubic polynomial, irrespectively of the value of $\beta$.

In essence we follow the method used by Ahmadi in~\cite{Ahmadi:Carlitz} for quadratic transformations,
adjusting his notation to the present more complex cubic setting.
Let $\mathcal{T}(m,q)$
be the set of elements of $\F_{q^m}$ which do not belong to any proper subfield.
In other words, $\mathcal{T}(m,q)$ consists of the elements of $\F_{q^m}$ which have degree $m$ over $\F_q$.
According to Lemma~\ref{Capelli}, we have $|\mathcal{I}(g,h,n,q)|=|U(n,q)|/n$, where
\[
U(m,q) = \{\beta \in \mathcal{T}(m,q): \text{$g(x)-\beta h(x)$ is irreducible over $\F_{q^m}$} \}.
\]
An application of M\"obius inversion will reduce the calculation of $|U(n,q)|$ to finding the cardinalities of the larger sets
\[
\overline{U}(m,q)=\{\beta \in\F_{q^m}:\text{$g(x)-\beta h(x)$ is irreducible over $\F_{q^m}$} \}.
\]
If $\beta \in \mathcal{T}(m,q)$ with $m$ a divisor of $n$ then $\beta\in\overline{U}(n,q)$ precisely when $n/m$ is not a multiple of three
(otherwise $\F_{q^n}$ contains a cubic extension of $\F_{q^m}$, where $g(x)-\beta h(x)$ has a root).
Hence
\begin{equation*}
|\overline{U}(n,q)|
=\sum_{d\mid n,\ 3\nmid d} |U(n/d,q)|,
\end{equation*}
and M\"obius inversion yields
\begin{equation}\label{eq:Moebius}
|\mathcal{I}(g,h,n,q)|=
\frac{1}{n}\sum_{d\mid n,\ 3\nmid d}\mu(d) |\overline{U}(n/d,q)|.
\end{equation}

We now compute the cardinalities of the sets $|\overline{U}(n/d,q)|$.

\begin{lemma}\label{lemma:Ubar}
We have
$|\overline{U}(n,q)|=(2q^n+A-B-2C-2D)/3$,
where
\begin{itemize}
\item
$A$ is the number of distinct roots of $h(x)$ in $\F_{q^n}$;
\item
$B$ is the number of $\beta\in\F_{q^n}$ for which $g(x)-\beta h(x)$ has a double root and a simple root in $\F_{q^n}$;
\item
$C$ is the number of $\beta\in\F_{q^n}$ for which $g(x)-\beta h(x)$ has a triple root in $\F_{q^n}$;
\item
$D$ is the number of $\beta\in\F_{q^n}$ for which $g(x)-\beta h(x)$ has a quadratic irreducible factor in $\F_{q^n}[x]$.
\end{itemize}
\end{lemma}

\begin{proof}
Because reducibility of a cubic polynomial over a field is equivalent to the polynomial having a root in that field,
we have $|\overline{U}(n,q)|=q^n-|V(n,q)|$, where
\[
V(n,q)=\{\beta \in\F_{q^n}: \text{$g(\gamma)=\beta h(\gamma)$ for some $\gamma\in\F_{q^n}$}\}.
\]
To compute $|V(n,q)|$ we perform a double counting on the set
\[
W(n,q)=\{(\gamma,\beta)\in\F_{q^n}\times\F_{q^n}: g(\gamma)=\beta h(\gamma) \}.
\]
Note that $g(x)$ and $h(x)$ have no common root because they are coprime.
On the one hand, for every $\gamma\in\F_{q^n}$ which is not a root of $h(x)$ there is a unique $\beta \in\F_{q^n}$
such that $g(\gamma)=\beta h(\gamma)$.
Hence $|W(n,q)|=q^n-A$, where $A$ is the number of distinct roots of $h(x)$ in $\F_{q^n}$.
On the other hand, counting in terms of $\beta$ and according to the splitting pattern of
$g(x)-\beta h(x)$ we find
\[
|W(n,q)|=3|V(n,q)|-B-2C-2D.
\]
Comparing the two expressions for $|W(n,q)|$ yields the desired conclusion.
\end{proof}

To proceed further we need a convenient way of computing the quantities involved in Lemma~\ref{lemma:Ubar}, especially $D$.
The discriminant
$
\Delta(\beta)=\Disc_x\bigl(g(x)-\beta h(x)\bigr)
$
will help to distinguish between the various factorization patterns of the cubic polynomial $g(x)-\beta h(x)$ as $\beta$ varies.
In particular, $\Delta(\beta)$ vanishes precisely when $g(x)-\beta h(x)$ has a multiple root.
Because a multiple root for a cubic polynomial necessarily lies in the field of coefficients,
$B+C$ equals the number of $\beta\in\F_{q^n}$ for which $\Delta(\beta)=0$.
Note that $\Delta(\beta)$ has degree at most four,
and its roots are the finite branch points of the rational expression $R(x)$.

When $g(x)-\beta h(x)$ is separable, in the modern meaning of having distinct roots in a splitting field,
information on its factorization is provided by its Galois group.
We view that as a group of permutations of the roots
(induced by automorphism of the splitting field), hence as a subgroup of the symmetric group $S_n$
once the roots have been arbitrarily numbered.
In characteristic different from two we have the following classical discriminant criterion.

\begin{theorem}
A separable polynomial of degree $n$, over a field $K$ of characteristic not two,
has Galois group over $K$ contained in the alternating group $A_n$ if and only if its discriminant is a square in $K$.
\end{theorem}

For cubic polynomials the following substitute is available,
which avoids the restriction on the characteristic.

\begin{theorem}\label{thm:Conrad}
A separable cubic polynomial $x^3+ax^2+bx+c$, over an arbitrary field $K$,
has Galois group over $K$ contained in the alternating group $A_3$ if and only if its {\em quadratic resolvent}
\[
x^2+(ab-3c)x+(a^3c+b^3+9c^2-6abc)
\]
is reducible over $K$.
\end{theorem}

This was given as an exercise in~\cite[page~53]{Kaplansky}, stated there in the special case where $a=0$.
Full details may be found in Keith Conrad's unpublished expository paper~\cite[Theorem~2.3]{Conrad}.
The essence of a proof is that, if $x_1,x_2,x_3$ are the roots of $x^3+ax^2+bx+c$ in a splitting field,
ordered in some arbitrary way, then the roots of the quadratic resolvent are
$\sum_{i=1}^3x_i^2x_{i\pm 1}$, with indices viewed modulo $3$,
and those are interchanged by odd permutations of  $x_1,x_2,x_3$.
Because the quadratic resolvent turns out to have the same discriminant as the cubic polynomial, this criterion leads back
to the more familiar discriminant criterion when the characteristic of the field is different from two.

\begin{lemma}\label{lemma:odd}
We have
$|\overline{U}(n,q)|=(N+A-C)/3$,
where $A$ and $C$ are as in Lemma~\ref{lemma:Ubar}, and
\begin{itemize}
\item
$N$ is the number of solutions in $\F_{q^n}\times\F_{q^n}$ of the equation $x^2+s(\beta)x+t(\beta)=0$
(in the unknowns $x$ and $\beta$),
where the left-hand side is the quadratic resolvent of $g(x)-\beta h(x)$.
\end{itemize}
If $q$ is odd, then
\begin{itemize}
\item
$N$ is the number of solutions in $\F_{q^n}\times\F_{q^n}$ of the equation $y^2-\Delta(\beta)=0$
(in the unknowns $y$ and $\beta$),
where
$\Delta(\beta)=\Disc_x\bigl(g(x)-\beta h(x)\bigr)$.
\end{itemize}
\end{lemma}

\begin{proof}
We have already noted that $B+C$ equals the number of $\beta\in\F_{q^n}$ for which $\Delta(\beta)=0$,
which is equivalent to the quadratic resolvent having a double root in $\F_{q^n}$.

Now consider those $\beta\in\F_{q^n}$ which are not roots of the discriminant,
whence both $g(x)-\beta h(x)$ and its quadratic resolvent are separable.
According to Theorem~\ref{thm:Conrad}, the quadratic resolvent
is irreducible over $\F_{q^n}$
if and only if the Galois group of $g(x)-\beta h(x)$ over $\F_{q^n}$
is not contained in the alternating group $A_3$.
Because Galois groups over a finite field are cyclic (hence the whole $S_3$ cannot be a Galois group),
this occurs precisely when the Galois group is a cyclic group of order two, interchanging two of the roots and fixing the third root.
Hence this is equivalent with $g(x)-\beta h(x)$ having a quadratic irreducible factor in $\F_{q^n}[x]$.
Therefore, $D$ equals the number of $\beta\in\F_{q^n}$ for which the quadratic resolvent is irreducible over $\F_{q^n}$, and hence has no roots in it.

As a consequence, the remaining case, where the quadratic resolvent has two distinct roots in $\F_{q^n}$,
occurs for precisely $q^n-B-C-D$ values of $\beta\in\F_{q^n}$.

Altogether, we find $N=(B+C)+2(q^n-B-C-D)=2q^n-B-C-2D$, and combining this with Lemma~\ref{lemma:Ubar} proves the first assertion.

The second assertion, for odd $q$, is a consequence of the former because the equation $y^2-\Delta(\beta)=0$ is equivalent to the
equation involving the quadratic resultant via the substitution $x=y+s(\beta)/2$.
\end{proof}

In the rest of the paper we will apply the method described in this section to compute
$|\mathcal{I}(g,h,n,q)|$ in various cases, thus applying Lemma~\ref{lemma:odd} followed by M\"obius inversion as in Equation~\eqref{eq:Moebius}.
The equation $x^2+s(\beta)x+t(\beta)=0$ of Lemma~\ref{lemma:odd},
considered over the algebraic closure $\barF_q$, represents an algebraic curve of degree at most four.
In various instances it will be possible to count the number of $\F_{q^n}$-rational points of the curve explicitly.
When interested only in odd characteristic we will use the simpler equation $y^2-\Delta(\beta)=0$ of Lemma~\ref{lemma:odd}.

Because of the final appeal to M\"obius inversion, our results will generally acquire a more compact appearance when expressed in terms of the function
\begin{equation}\label{eq:I}
I(n,q)=\frac{1}{n}\sum_{d\mid n}\mu(d)q^{n/d},
\end{equation}
which counts the number of irreducible monic polynomials of degree $n$ in $\F_{q}[x]$,
for any positive integer $n$.
For $m$ a (positive) divisor of $n$ we have
\[
\frac{1}{n}\sum_{md\mid n}\mu(d)q^{n/d}
=\frac{1}{m}I(n/m,q^m).
\]
This remains valid even when $m$ does not divide $n$ if we stipulate that $I(n,q)$ should take the value zero whenever $n$ is not a positive integer.
In particular, writing $n_3$ for the largest power of $3$ which divides $n$,
in applications of Equation~\eqref{eq:Moebius} we will take advantage of expressing our results in terms of the function $I(n',q')$ using
\begin{equation}\label{eq:Moebius_compact}
\frac{1}{n}\sum_{d\mid n,\ 3\nmid d}\mu(d)q^{n/d}
=\frac{1}{n}\sum_{n_3d\mid n}\mu(d)q^{n/d}
=\frac{1}{n_3}I(n/n_3,q^{n_3}).
\end{equation}

To avoid several case distinctions in our results,
we will allow ourselves to use the expression for $I(n,q)$ given in Equation~\eqref{eq:I},
which is formally a polynomial in $q$, for values of $q$ which are not prime powers,
irrespectively of the lack of combinatorial interpretation.
In particular, $I(n,1)$ vanishes unless $n=1$, in which case $I(1,1)=1$.
We will also need $I(n,-1)$, which vanishes except when $n\le 2$,
in which cases $I(1,-1)=-1$ and $I(2,-1)=2$.

In some of our proofs it will be useful to consider $I(n,-q)$, but then
\begin{equation}\label{eq:I+-q}
\begin{aligned}
I(n,q)+I(n,-q)
&=
\frac{1}{n}\sum_{d\mid n}\mu(d)\bigl(1+(-1)^{n/d}\bigr)q^{n/d}
\\&=
\frac{2}{n}\sum_{2d\mid n}\mu(d)(q^2)^{n/2d}
=I(n/2,q^2),
\end{aligned}
\end{equation}
to be read as zero when $n$ is odd according to our convention.

\section{Counting polynomials in large characteristics}\label{sec:three_ram_points}

In this section we will use Lemma~\ref{lemma:odd} to compute $|\mathcal{I}(g,h,n,q)|$
for any cubic rational expression $R(x)=g(x)/h(x)$ over $\F_q$ having at most three ramification points.
According to Lemma~\ref{lemma:invariance} it suffices to do that for the representatives of
equivalence classes of such expressions $R(x)$
given in Theorem~\ref{thm:cubic_expr_finite} for characteristic at least five.
However, we will cover smaller characteristics as well in this section when possible and convenient.

The case of $R(x)=x^3$ is a special case of a result of Cohen in~\cite[Theorem~3]{Cohen:irreducible},
who counted irreducible polynomials of the form $f(x^r)$.
For completeness we provide a proof of this special case of Cohen's result, as a first illustration of
our method described in Section~\ref{sec:tools}.
Over a finite field $\F_q$ of characteristic three any polynomial $f(x^3)$ is the cube of a polynomial, hence is reducible,
and so $\mathcal{I}(x^3,1,n,q)$ is empty in this case.

\begin{theorem}\label{thm:x^3}
If $q\equiv 1\pmod{3}$, then
\[
|\mathcal{I}(x^3,1,n,q)|
=\frac{2}{3n_3}\bigl(I(n/n_3,q^{n_3})-I(n/n_3,1)\bigr).
\]
If $q\equiv -1\pmod{3}$, then
\[
|\mathcal{I}(x^3,1,n,q)|
=\frac{1}{3n_3}\bigl(I(n/2n_3,q^{2n_3})-I(n/2n_3,1)\bigr).
\]
\end{theorem}

\begin{proof}
We prepare for an application of Lemma~\ref{lemma:odd}, in the form with the quadratic resolvent, as that allows characteristic two as well.
We have $A=0$ because $h(x)=1$ has no roots, and $C=1$
because $x^3-\beta$ has a triple root only when $\beta=0$
(as the characteristic is not three here).

The quantity $N$ equals the number of solutions in $\F_{q^n}\times\F_{q^n}$ of the equation
$x^2+3\beta x+9\beta^2=0$,
where the left-hand side is the quadratic resolvent of $x^3-\beta$.
The equation represents a singular conic, consisting of two lines through the origin,
which are defined over $\F_{q^n}$ if and only if $q^n\equiv 1\pmod{3}$.
Hence we have
$N=2q^n-1$ if $q^n\equiv 1\pmod{3}$, and
$N=1$ if $q^n\equiv -1\pmod{3}$.

Hence according to Lemma~\ref{lemma:odd} we have
$|\overline{U}(n,q)|=2(q^n-1)/3$
if $q^n\equiv 1\pmod{3}$,
and $|\overline{U}(n,q)|=0$ otherwise.
Both cases can be stated together as
\[
|\overline{U}(n,q)|=\bigl(q^n+(\varepsilon q)^n-1-\varepsilon^n\bigr)/3,
\]
using the Legendre symbol
$\varepsilon=\left(\frac{q}{3}\right)$, hence $\varepsilon=\pm1$
according to whether $q\equiv\pm1\pmod{3}$.
M\"obius inversion as given in Equation~\eqref{eq:Moebius} now yields
\begin{equation}\label{eq:x^3}
\begin{aligned}
|\mathcal{I}&(x^3,1,n,q)|
=
\frac{1}{3n}\sum_{d\mid n,\ 3\nmid d}\mu(d)
\bigl(q^{n/d}+(\varepsilon q)^{n/d}-1-\varepsilon^{n/d}\bigr).
\\&=
\frac{1}{3n_3}\bigl(
I(n/n_3,q^{n_3})+I(n/n_3,\varepsilon q^{n_3})-I(n/n_3,1)-I(n/n_3,\varepsilon)
\bigr),
\end{aligned}
\end{equation}
where we have used Equation~\eqref{eq:Moebius_compact} to get the final expression.
The desired conclusions follow, using Equation~\eqref{eq:I+-q} in case $q\equiv -1\pmod{3}$.
\end{proof}

\begin{rem}\label{rem:Cohen}
The results of Theorem~\ref{thm:x^3} can be stated together as
\[
|\mathcal{I}(x^3,1,n,q)|
=\frac{2}{3kn_3}\bigl(I(n/kn_3,q^{kn_3})-I(n/kn_3,1)\bigr),
\]
where $k$ be the smallest positive integer such that $3\mid q^k-1$,
hence $k\in\{1,2\}$ and $k\equiv q\pmod{3}$.
This statement is more akin to the statement for the general case of $x^r$ given in~\cite[Theorem~3]{Cohen:irreducible}.
\end{rem}

According to Theorem~\ref{thm:cubic_expr_finite}, any cubic rational expression $R(x)$, over $\F_q$ of characteristic at least five,
having only two ramification points (over $\barF_q$)
is equivalent to either $x^3$, or $(x^3+3\sigma x)/(3x^2+\sigma)$, where $\sigma$ is any nonsquare in $\F_q$.
The following result deals with the latter case.

\begin{theorem}\label{thm:two_ram_points}
Let $\F_q$ have characteristic at least five,
and let $\sigma\in\F_q$ be a nonsquare.
If $q\equiv 1\pmod{3}$, then
\[
|\mathcal{I}(x^3+3\sigma x,3x^2+\sigma,n,q)|
=\frac{1}{3n_3}\bigl(I(n/2n_3,q^{2n_3})-I(n/2n_3,1)\bigr).
\]
If $q\equiv -1\pmod{3}$, then
\[
|\mathcal{I}(x^3+3\sigma x,3x^2+\sigma,n,q)|
=\frac{2}{3n_3}\bigl(I(n/n_3,q^{n_3})-I(n/n_3,-1)\bigr).
\]
\end{theorem}

\begin{proof}
The polynomial $(x^3+3\sigma x)-\beta(3x^2+\sigma)$ has quadratic resolvent
$x^2-6\beta\sigma x+9\sigma(3\beta^4-5\sigma\beta^2+3\sigma^2)$,
but because the characteristic is odd it will be more convenient to work directly with its discriminant
$\Delta(\beta)=-108\sigma(\beta^2-\sigma)^2$,
and apply Lemma~\ref{lemma:odd} in its second formulation.

The number $A$ of roots of $3x^2+\sigma$ in $\F_{q^n}$ is given by
$A=1+(-\varepsilon)^n$,
in terms of the Legendre symbol
$\varepsilon=\left(\frac{q}{3}\right)$.
Because the polynomial $(x^3+3\sigma x)-\beta(3x^2+\sigma)$ (as a polynomial in $x$, for a given value of $\beta$) has a triple root in $\F_{q^n}$
if and only if $\beta^2=\sigma$, we have $C=1+(-1)^n$.

The number $N$ of solutions in $\F_{q^n}\times\F_{q^n}$ of the quartic equation
$y^2+108\sigma(\beta^2-\sigma)^2=0$
can be computed by assigning values for $\beta$ in $\F_{q^n}$ and then solving for $y$.
For a given value of $\beta$, the equation has one solution if $\beta^2=\sigma$, and $1+(-\varepsilon)^n$ solutions otherwise.
Hence
$N=2+(q^n-2)\bigl(1+(-\varepsilon)^n\bigr)$
if $n$ is even, and
$N=q^n\bigl(1+(-\varepsilon)^n\bigr)$
if $n$ is odd.
In a single formula,
$N=q^n+(-\varepsilon q)^n-(-\varepsilon)^n-\varepsilon^n$.

According to Lemma~\ref{lemma:odd} we find
\[
|\overline{U}(n,q)|=\bigl(q^n+(-\varepsilon q)^n-(-1)^n-\varepsilon^n\bigr)/3.
\]
Applying M\"obius inversion as in Equation~\eqref{eq:Moebius}, and then Equation~\eqref{eq:Moebius_compact}, we find
\begin{equation*}
\begin{aligned}
|\mathcal{I}&(x^3+3\sigma x,3x^2+\sigma,n,q)|
\\&=
\frac{1}{3n_3}\bigl(
I(n/n_3,q^{n_3})+I(n/n_3,-\varepsilon q^{n_3})-I(n/n_3,-1)-I(n/n_3,\varepsilon)
\bigr).
\end{aligned}
\end{equation*}
The desired conclusions follow, using Equation~\eqref{eq:I+-q} in case $q\equiv 1\pmod{3}$.
\end{proof}

Note that if $\sigma$ is a nonzero square in $\F_q$ (for example $\sigma=1$), rather than a nonsquare, then
$|\mathcal{I}(x^3+3\sigma x,3x^2+\sigma,n,q)|
=|\mathcal{I}(x^3,1,n,q)|$,
because then the cubic rational expression
$(x^3+3\sigma x)/(3x^2+\sigma)$ is equivalent to $x^3$ over $\F_q$.
This suggests an alternate way of recovering Theorem~\ref{thm:x^3} (for odd $q$),
by adapting the above proof to $\sigma$ a nonzero square in $\F_q$ instead of a nonsquare.

\begin{rem}\label{rem:I_empty}
According to Theorem~\ref{thm:x^3}, the set $\mathcal{I}(x^3,1,n,q)$ is empty if $q^n\equiv -1\pmod{3}$.
A more direct reason is that $f(x^3)$ cannot be irreducible in that case
because $x^3$ induces a permutation of $\F_{q^n}$,
which of course commutes with the map $\alpha\mapsto\alpha^q$.
Consequently, writing
$f(x)=(x-\alpha)(x-\alpha^q)\cdots(x-\alpha^{q^{n-1}})$
for some $\alpha\in\F_{q^n}$,
the polynomial
$(x-\beta)(x-\beta^q)\cdots(x-\beta^{q^{n-1}})$,
is a factor of $f(x^3)$,
where $\beta$ is the unique cubic root of $\alpha$ in $\F_{q^n}$.
Here we may even be more explicit, as $\beta=\alpha^{(2q^n-1)/3}$.

Similarly, according to Theorem~\ref{thm:two_ram_points} the set
$\mathcal{I}(x^3+3\sigma x,3x^2+\sigma,n,q)$ is empty if $q\equiv 1\pmod{3}$ and $n$ is odd.
This can also be explained directly noting that under those conditions the expression
$R(x)=(x^3+3\sigma x)/(3x^2+\sigma)$
induces a permutation of $\F_{q^n}$,
as can be seen by solving the equation $R(x)=R(y)$ and finding that for $x\neq y$ it is equivalent to
$-3(xy-\sigma)^2=\sigma(x-y)^2$.
Consequently, one of the factors of $f_R(x)$ over $\F_q$, with $f(x)$ as before, has a root $\gamma\in\F_{q^n}$ determined by
$R(\gamma)=\alpha$.
\end{rem}

Now we turn our attention to cubic rational expressions $R(x)$ with three ramification points.
According to Theorem~\ref{thm:cubic_expr_finite} those are equivalent to either $x^3-3x$ or $x^3-3\sigma x$, where $\sigma$ is a nonsquare in $\F_q$.
However, this distinction will actually be immaterial for our count of irreducible polynomials, and so we will treat both cases together.
The special case of the following result where $R(x)=x^3-3x$ is~\cite[Theorem~4]{Chu:cubic}.

\begin{theorem}\label{thm:three_ram_points}
Let $\F_q$ have characteristic at least five, and let $\tau$ be a nonzero element of $\F_q$.
Then
\[
|\mathcal{I}(x^3-\tau x,1,n,q)|
=\frac{1}{3n_3}\bigl(I(n/n_3,q^{n_3})-I(n/n_3,\varepsilon)\bigr).
\]
where $\varepsilon=\left(\frac{q}{3}\right)$ is a Legendre symbol.
\end{theorem}

\begin{proof}
Because the characteristic is odd we will apply Lemma~\ref{lemma:odd} in the form with the discriminant, which in case of
the polynomial $x^3-\tau x-\beta$ is $\Delta(\beta)=4\tau^3-27\beta^2$.
Both quantities $A$ and $C$ vanish here, and it remains to compute the number $N$ of solutions in $\F_{q^n}\times\F_{q^n}$ of the equation
$y^2+27\beta^2-4\tau^3=0$ (in the unknowns $y$ and $\beta$).
Because this represents a non-singular conic, it will have precisely $q^n+1$ points in the projective plane over $\F_{q^n}$,
and subtracting the $1+\varepsilon^n$ points at infinity we find $N=q^n-\varepsilon^n$.
Hence
\[
|\overline{U}(n,q)| = \frac{q^n-\varepsilon^n}{3},
\]
and the conclusion follows after M\"obius inversion.
\end{proof}

For cubic rational expressions $R(x)=g(x)/h(x)$ with four rational points
we can only produce an estimate for $|\mathcal{I}(g,h,n,q)|$.
This is because $|\mathcal{I}(g,h,n,q)|$ turns out to depend on the number of $\F_q$-rational points of a curve of genus one associated with $R(x)$.
Because a cubic rational expression $R(x)$ over a field of characteristic two cannot have four ramification points
(see~\cite[Section~4]{MatPiz:cubic-maps}, for example),
in the following result $\F_q$ has necessarily odd characteristic.

\begin{theorem}\label{thm:four_ram_points}
Let $R(x)=g(x)/h(x)$ be a cubic rational expression over $\F_q$ with four ramification points (over $\barF_q$).
Then $|\mathcal{I}(g,h,n,q)|= q^n/3n + O(q^{n/2}/n)$.
\end{theorem}

\begin{proof}
An application of Lemma~\ref{lemma:odd} involves the number $N$ of solutions in $\F_{q^n}\times\F_{q^n}$
of the equation $y^2-\Delta(\beta)=0$.
Recall that $\Delta(\beta)$ is a polynomial of degree at most four,
and that its roots are precisely the finite ramification points of $R(x)$.
If $\infty$ is a ramification point of $R(x)$, then $h(x)$ has degree less than three,
and then one sees that $\Delta(\beta)$ has degree at most three.
Because we are assuming that $R(x)$ has four ramification points, $\Delta(\beta)$ is a cubic polynomial with distinct roots.
However, if $\infty$ is not a ramification point of $R(x)$, then $\Delta(\beta)$ is a quartic polynomial with distinct roots.

Now consider the plane curve with affine equation $y^2=\Delta(\beta)$,
which is absolutely irreducible in either case.
If $\Delta(\beta)$ has degree three, then the projective version of the curve is a non-singular curve of genus one,
whose number of its $\F_{q^n}$-rational points equals $N+1$.
The Hasse-Weil bound then gives
\[
|N-q^n|\le 2\sqrt{q^n}.
\]
Together with Equation~\eqref{eq:Moebius} this yields $|\mathcal{I}(g,h,n,q)|= q^n/3n + O(q^{n/2}/n)$, as desired.

If $\Delta(\beta)$ has degree four, then the projective version of the curve has a singularity in the point at infinity.
According to~\cite[Corollary~2.5]{Aubry-Perret} we have the weaker bound
\[
|N-q^n|\le 6\sqrt{q^n}.
\]
Here the coefficient $6$ results from having replaced the (geometric) genus of the curve (which is still $1$ according to Pl\"ucker's formula)
with its arithmetic genus (which is $(d-1)(d-2)/2=3$, where $d=4$ is the degree of the curve).
Together with Equation~\eqref{eq:Moebius} this yields the same asymptotic estimate of
$|\mathcal{I}(g,h,n,q)|$.
\end{proof}

Because of the exact formulas of Theorem~\ref{thm:three_ram_points},
$|\mathcal{I}(g,h,n,q)|$ satisfies the same estimate of Theorem~\ref{thm:four_ram_points}
when $R(x)=g(x)/h(x)$ has three ramification points.
However, according to Theorems~\ref{thm:x^3} and~\ref{thm:two_ram_points} the asymptotic estimate becomes twice as large in case $R(x)$ has just two ramification points.

\section{Results in characteristic three}\label{sec:char_three}

A classification of cubic rational expressions $R(x)$ up to equivalence, analogous to that of
Theorem~\ref{thm:cubic_expr_finite}
but over a finite field of characteristic three,
was also obtained in~\cite{MatPiz:cubic-maps},
under the same restriction that $R(x)$ has at most three ramification points.
It reads as follows.

\begin{theorem}[\cite{MatPiz:cubic-maps}, Theorem~18]\label{thm:cubic_expr_finite_3}
Let $\F_q$ be a finite field of characteristic three,
and let $\sigma\in\F_q$ be a nonsquare.
Then any cubic rational expression $R(x)$ over $\F_q$ with at most three ramification points is equivalent to
either $x^3+x^2$, or $x^3+x$, or $x^3+\sigma x$.
\end{theorem}

The expression $x^3+x^2$ has two ramification points, namely, $\infty$ with index three and $0$ with index two.
Every cubic rational expression, over $\F_q$ of characteristic three, with precisely two ramification points, is equivalent to $x^3+x^2$.
Each of the expressions $x^3+x$ and $x^3+\sigma x$ has $\infty$ as unique ramification point, with index three.
Every cubic rational expression, over $\F_q$ of characteristic three, with precisely one ramification point, is equivalent to one of those.
Also, those two become equivalent to each other over $\F_{q^2}$.
Note that the assumption of having at most three ramification points in Theorem~\ref{thm:cubic_expr_finite_3}
excludes the inseparable cubic rational expressions, which are ramified everywhere.
Those are easily seen to be all equivalent to $x^3$.

Now we apply the method of Section~\ref{sec:char_three} to compute $|\mathcal{I}(g,h,n,q)|$
for the various representatives of Theorem~\ref{thm:cubic_expr_finite_3}.
As mentioned before Theorem~\ref{thm:x^3}, the set $\mathcal{I}(x^3,1,n,q)$ is empty when $q$ is a power of three.
If $R(x)$ has four distinct ramification points we generally only have the estimate of Theorem~\ref{thm:four_ram_points},
so it remains to cover the cases where $R(x)$ equals $x^3+x^2$, or $x^3-x$, or $x^3-\sigma x$,
where $\sigma$ is a nonsquare in $\F_q$.

\begin{theorem}\label{thm:char_three_(3,2)}
Over the field $\F_q$ of characteristic three we have
\[
|\mathcal{I}(x^3+x^2,1,n,q)|
=\frac{1}{3n_3}I(n/n_3,q^{n_3}).
\]
\end{theorem}

\begin{proof}
Both quantities $A$ and $C$ of Lemma~\ref{lemma:odd} vanish, and the polynomial $x^3+x^2-\beta$ has discriminant $\Delta(\beta)=\beta$.
The number $N$ of solutions in $\F_{q^n}\times\F_{q^n}$ of the equation $y^2-\beta=0$ equals $q^n$.
Lemma~\ref{lemma:odd} yields $|\overline{U}(n,q)| = q^n/3$,
and the conclusion follows by M\"obius inversion as given by Equation~\eqref{eq:Moebius}.
\end{proof}

\begin{theorem}\label{thm:char_three_(3)}
Let $\F_q$ have characteristic three,
and let $\sigma\in\F_q$ be a nonsquare.
Then
\[
|\mathcal{I}(x^3-x,1,n,q)|
=\frac{2}{3n_3}I(n/n_3,q^{n_3}),
\]
and
\[
|\mathcal{I}(x^3-\sigma x,1,n,q)|
=\frac{1}{3n_3}I(n/2n_3,q^{2n_3}).
\]
\end{theorem}

\begin{proof}
Both quantities $A$ and $C$ of Lemma~\ref{lemma:odd} vanish.
We may deal with both cases together by taking $\tau\in\{1,\sigma\}$
(or $\tau$ any nonzero element of $\F_q$).
The polynomial $x^3-\tau x-\beta$ has discriminant $\Delta(\beta)=\tau^3$, hence independent of $\beta$.
The number $N$ of solutions in $\F_{q^n}\times\F_{q^n}$ of the equation $y^2-\tau^3=0$
(in the unknowns $y$ and $\beta$, with the latter not appearing)
is given by $N=\bigl(1+\eta_q(\tau)^n\bigr)q^n$, where $\eta_q$ is the quadratic character of $\F_q$.
Hence
\[
|\overline{U}(n,q)|=\frac{q^n+\eta_q(\tau)^n q^n}{3}
\]
according to Lemma~\ref{lemma:odd}, and M\"obius inversion yields
\[
|\mathcal{I}(x^3-\tau x,1,n,q)|
=
\frac{1}{3n_3}\bigl(
I(n/n_3,q^{n_3})+I(n/n_3,\eta_q(\tau) q^{n_3})
\bigr).
\]
The desired conclusions follow, using Equation~\eqref{eq:I+-q} in case $\tau=\sigma$.
\end{proof}

As in Remark~\ref{rem:I_empty}, the fact that $\mathcal{I}(x^3-\sigma x,1,n,q)$ is empty
when $n$ is odd in Theorem~\ref{thm:char_three_(3)}
is also due to $x^3-\sigma x$ inducing a permutation of $\F_{q^n}$ in that case.

\section{Results in characteristic two}\label{sec:char_two}

A classification of cubic rational expressions over finite fields of characteristic two, analogous to Theorem~\ref{thm:cubic_expr_finite},
was also obtained in~\cite{MatPiz:cubic-maps}.
However, because no cubic rational expression in characteristic two can have more than three ramification points (see~\cite[Section~4]{MatPiz:cubic-maps}),
in this case the classification covers all cubic rational expressions, with no restriction.

\begin{theorem}[\cite{MatPiz:cubic-maps}, Theorem~19]\label{thm:cubic_expr_finite_2}
Let $\F_q$ be a finite field of characteristic two,
and let $\sigma\in\F_q$ be an element of absolute trace $1$.
If $q$ is a square then let $\theta$ be a non-cube in $\F_q$.
Then any cubic rational expression $R(x)$ over $\F_q$  is equivalent to precisely one of the following:
\begin{itemize}
\item[(i)]
$x^3$, with ramification indices $(3,3)$;
\item[(ii)]
$\displaystyle
\frac{x^3+\sigma x+\sigma}{x^2+x+\sigma+1}$,
with ramification indices $(3,3)$;
\item[(iii)]
$x^3 + x^2$,
with ramification indices $(3,2)$;
\item[(iv)]
$(x^3+1)/x$
or, only in case $q$ is a square,
$(x^3+\theta)/x$ or $(x^3+\theta^2)/x$;
all these have ramification indices $(2)$;
\item[(v)]
$(x^3 + c)/(x+c)$,
with ramification indices $(2,2)$,
for a unique $c\in\F_q\setminus\F_2$;
\item[(vi)]
$\displaystyle
\frac{x^3 +bx^2 +\sigma x +(b+1)\sigma}
{x^2 +x +b+1+\sigma}$,
with ramification indices $(2,2)$,
for a unique $b\in\F_q\setminus\F_2$.
\end{itemize}
\end{theorem}

The cubic expressions of cases~(i) and~(ii) are equivalent over $\F_{q^2}$.
The expressions of case~(vi) is equivalent to that of case~(v) with $c=1/b^2$.
If $q$ is a square then the three expressions of case~(iv) become equivalent over $\F_{q^3}$.
Thus, over the algebraic closure $\barF_q$ the above six cases collapse to three isolated cases and one parametric family,
as in~\cite[Theorem~10]{MatPiz:cubic-maps}.

In this section we will apply the general method of Section~\ref{sec:tools} to the calculation of
$|\mathcal{I}(g,h,n,q)|$
in each case of Theorem~\ref{thm:cubic_expr_finite_2}.
We will find explicit formulas in cases~(i), (ii), (iii), but only an estimate in the remaining cases.

Case~(i) is covered by Theorem~\ref{thm:x^3}.
The following result deals with case~(ii), and produces exactly the same formulas as in Theorem~\ref{thm:two_ram_points}.

\begin{theorem}\label{thm:char_two_(3,3)}
Let $\F_q$ have characteristic two, and let $\sigma\in\F_q$ have absolute trace $1$.
If $q\equiv 1\pmod{3}$, then
\[
|\mathcal{I}(x^3+\sigma x+\sigma,x^2+x+\sigma+1,n,q)|
=\frac{1}{3n_3}\bigl(I(n/2n_3,q^{2n_3})-I(n/2n_3,1)\bigr).
\]
If $q\equiv -1\pmod{3}$, then
\[
|\mathcal{I}(x^3+\sigma x+\sigma,x^2+x+\sigma+1,n,q)|
=\frac{2}{3n_3}\bigl(I(n/n_3,q^{n_3})-I(n/n_3,-1)\bigr).
\]
\end{theorem}

\begin{proof}
We prepare for an application of Lemma~\ref{lemma:odd} in the form with the quadratic resolvent.
The number $A$ of roots of $x^2+x+\sigma+1$ in $\F_{q^n}$
equals $2$ or $0$ according to whether the absolute trace of $\sigma+1$
as an element of $\F_{q^n}$ equals $0$ or $1$.
Consequently, we have
$A=1+(-\varepsilon)^n$
in terms of the Legendre symbol
$\varepsilon=\left(\frac{q}{3}\right)$.
The polynomial
$g(x)-\beta h(x)=x^3+\beta x^2+(\beta+\sigma)x+(\beta\sigma+\beta+\sigma)$,
as a polynomial in $x$, for a given value of $\beta$, has a triple root in $\F_{q^n}$
if and only if $\beta^2+\beta=\sigma$.
Hence we have $C=1+(-1)^n$.

The quadratic resolvent of $g(x)-\beta h(x)$ is
\[
x^2+(\beta^2+\beta+\sigma)x+
\bigl((\sigma+1)\beta^4+(\sigma+1)\beta^3+(\sigma^2+\sigma+1)\beta^2+\sigma^2\beta+(\sigma^3+\sigma^2)\bigr).
\]
Over any extension of $\F_q$ containing an element $\rho$ with $\rho^2+\rho=\sigma+1$,
the quadratic resolvent factorizes as
\[
\bigl(x+\rho\beta^2+\rho^2\beta+(\rho^3+1)\bigr)
\bigl(x+(\rho+1)\beta^2+(\rho^2+1)\beta+(\rho^3+\rho^2+\rho)\bigr).
\]
Equating this to zero yields the equation of the union of two non-singular conics,
which meet in the points
$(x,\beta)=(\tau^3,\tau)$, where $\tau^2+\tau=\sigma$
(hence $\tau$ is as in the proof of~\cite[Theorem~19]{MatPiz:cubic-maps}), and
$(x,\beta)=((\tau+1)^3,\tau+1)$.
Because $\sigma$ is an element of $\F_q$ with absolute trace $1$, the element $\tau$ belongs to $\F_{q^n}$
if and only if $n$ is even.

Furthermore, $\rho$ belongs to $\F_{q^n}$ unless $\varepsilon=1$ and $n$ is odd.
If $\rho\in\F_{q^n}$, then the number of solutions in $\F_{q^n}\times\F_{q^n}$ of the equation associated with the quadratic resolvent is given by
$N=2q^n-1-(-1)^n$,
having taken into account the possible intersections of the two (affine) conics.
However, if $\rho\not\in\F_{q^n}$, then the two conics are only defined over $\F_{q^{2n}}$ and they are conjugate over $\F_{q^n}$.
Because their two intersection points described earlier do not have coordinates in $\F_{q^n}$
in this case, we find $N=0$.
In a single formula encompassing all cases, we have
$N=q^n+(-\varepsilon q)^n-(-\varepsilon)^n-\varepsilon^n$.

Noting that $N$, $A$ and $C$ are given by identical expressions as in the proof of Theorem~\ref{thm:two_ram_points},
the conclusion follows exactly as in that proof.
\end{proof}

Once again, the fact that Theorem~\ref{thm:char_two_(3,3)}
reports no irreducible polynomials if $q\equiv 1\pmod{3}$ and $n$ is odd
can be explained as in Remark~\ref{rem:I_empty},
because $R(x)$ induces a permutation of $\F_{q^n}$ in that case.

The following result deals with case~(iii) of Theorem~\ref{thm:cubic_expr_finite_2}.

\begin{theorem}\label{thm:char_two_(3,2)}
If $q$ is a power of two, then
\[
|\mathcal{I}(x^3+x^2,1,n,q)|
=\frac{1}{3n_3}\bigl(I(n/n_3,q^{n_3})-I(n/n_3,\varepsilon)\bigr),
\]
where $\varepsilon=\left(\frac{q}{3}\right)$.
\end{theorem}

\begin{proof}
Both quantities $A$ and $C$ of Lemma~\ref{lemma:odd} vanish.
The quadratic resolvent of $x^3+x^2-\beta$
equals $x^2+\beta x+\beta^2+\beta$.
Equating this polynomial to zero yields the equation of a non-singular conic.
Over $\F_{q^n}$ the conic has $q^n+1$ projective points, of which $1+\varepsilon^n$ points at infinity.
Consequently, $N=q^n-\varepsilon^n$, hence
$|\overline{U}(n,q)|=(q^n-\varepsilon^n)/3$
according to Lemma~\ref{lemma:odd}, and the conclusion follows by M\"obius inversion as usual.
\end{proof}

In the remaining cases of Theorem~\ref{thm:cubic_expr_finite_2}, the method of Section~\ref{sec:tools} leads to counting
points on a certain curve of genus one, and hence we will only lead to an estimate for $|\mathcal{I}(g,h,n,q)|$.
We do that in the following result, which is analogous to the case of $R(x)$ with four ramification points in odd characteristic,
Theorem~\ref{thm:four_ram_points}.

\begin{theorem}\label{thm:Weil_char_2}
Let $R(x)=g(x)/h(x)$ be a cubic rational expression, over $\F_q$ of characteristic two, with no ramification point
having ramification index $3$.
Then $|\mathcal{I}(g,h,n,q)|= q^n/3n + O(q^{n/2}/n)$.
\end{theorem}

\begin{proof}
According to Theorem~\ref{thm:cubic_expr_finite_2}, the cubic rational expression $R(x)$ is equivalent, over $\F_{q}$, to one of those in cases~(iv), (v) or~(vi).
We consider each case in turn.

In case~(iv) we have $g(x)/h(x)=(x^3+c)/x$ with some nonzero $c\in\F_q$.
Setting the quadratic resolvent of
$g(x)-\beta h(x)$ equal to zero we find the equation
$x^2+cx=\beta^3+c^2$.
The projective version of the curve given by this equation is a non-singular curve of genus one.
Similarly to the proof of Theorem~\ref{thm:four_ram_points}, the Hasse-Weil bound applies and gives
\[
|N-q^n|\le 2\sqrt{q^n}.
\]
Together with Equation~\eqref{eq:Moebius} this yields $|\mathcal{I}(g,h,n,q)|= q^n/3n + O(q^{n/2}/n)$, as desired.

In case~(v), setting the quadratic resolvent of
$g(x)-\beta h(x)$ equal to zero we find the equation
\[
x^2+cx\beta+cx=\beta^3+c^2\beta^2+c^2,
\]
with $c\in\F_q\setminus\F_2$.
After checking that the projective version of the curve is a non-singular curve of genus one, one concludes as in the previous case.

Finally, in case~(vi), the quadratic resolvent leads to an equation of the form
\[
x^2+(\beta^2+\beta+\sigma)x=u(\beta),
\]
for a certain quartic polynomial $u(\beta)$ with leading coefficient $b+1+\sigma$,
where $b\in\F_q\setminus\F_2$.
Calculations here are more involved than in previous cases, but with a little help from any computer algebra system
one can show that the projective version of this equation represents an absolutely irreducible quartic curve,
with a singularity in the point at infinity.
As in the proof of Theorem~\ref{thm:four_ram_points}, one has
\[
|N-q^n|\le 6\sqrt{q^n},
\]
and again the desired conclusion follows after applying M\"obius inversion.
\end{proof}

\section{Transformation shift registers}\label{sec:TSR}

Linear feedback shift registers are used to generate sequences over a finite field in a range of applications,
from pseudo-random number generation to the design of stream ciphers, see~\cite[Chapter~8, Section~1]{LN}.
Transformation shift registers (TSRs)
can be considered a generalization of linear feedback shift registers.
We refer the reader to~\cite[Section~3]{Cohen+:TSR} for an introduction to TFRs,
and focus here on the connection between our results and counting TSRs.

There is a notion of irreducible TSR, and following notation in~\cite{Cohen+:TSR}
we let $\TSRI(m,n,q)$ denote the set of irreducible TSRs of order $n$ over $\F_{q^m}$.
In~\cite{Ram:shift_registers} Ram related the cardinality of $\TSRI(m,n,q)$ to counting irreducible polynomials of the form $f_R(x)$,
with $R(x)$ ranging over rational expressions of degree $n$ of a special form.
In the present paper's notation, and according to a further development in~\cite[Theorem~3.2]{Cohen+:TSR}, we have
\begin{equation}\label{eq:TSRI}
|\TSRI(m,n,q)|=
\frac{|\GL_m(\F_q)|}{q^m-1}
\sum_{h(x)}|\mathcal{I}(x^n,h(x),m,q)|,
\end{equation}
where the summation $\sum_{h(x)}$ runs over all $h(x)\in\F_q[x]$
having degree less than $n$ and satisfying $h(0)=1$.
Here $|\GL_m(\F_q)|=\prod_{i=0}^{m-1}(q^m-q^i)$ is the order of the general linear group.
This connection, together with Ahmadi's evaluation, in~\cite{Ahmadi:Carlitz}, of $|\mathcal{I}(g(x),h(x),m,q)|$ for the case where $g(x)/h(x)$ is quadratic,
was then used in~\cite[Theorem~3.4]{Cohen+:TSR} to obtain a new derivation of an explicit formula for $|\TSRI(m,2,q)|$,
originally found by Ram~\cite{Ram:shift_registers} by a different method.
We use a similar method to recover the formulas for $|\TSRI(m,3,q)|$ obtained by Jiang and Yang in~\cite{Jiang-Yang}.

Start with noting that if $h(x)$ has degree less than $n$ and satisfies $h(0)=1$ as above,
then $g(x)=(1/x)\circ\bigl(x^n/h(x)\bigr)\circ(1/x)$ is a monic polynomial of degree $n$ and with $g(0)=0$.
According Section~\ref{sec:equivalence} we then have
\[
|\mathcal{I}(x^n,h(x),m,q)|=|\mathcal{I}(g(x),1,m,q)|.
\]
In particular, when $n=3$ we have
\[
|\TSRI(m,3,q)|=
\frac{|\GL_m(\F_q)|}{q^m-1}
\sum_{a,b\in\F_q}|\mathcal{I}(x^3+ax^2+bx,1,m,q)|,
\]
which we now in a position to compute using our results of the previous sections.

\begin{theorem}\label{thm:TSRI}
If $q$ is not a power of three, then
\[
|\TSRI(m,3,q)|=
\frac{q|\GL_m(\F_q)|}{3m(q^m-1)}
\sum_{d\mid n,\ 3\nmid d}
\mu(d)\bigl(q+\varepsilon^{m/d}\bigr)\bigl(q^{m/d}-\varepsilon^{m/d}\bigr).
\]
If $q$ is a power of $3$, then
\[
|\TSRI(m,3,q)|=
\frac{(q-1)|\GL_m(\F_q)|}{3m(q^m-1)}
\sum_{d\mid n,\ 3\nmid d}
\mu(d)\left(q+\frac{3+(-1)^{m/d}}{2}\right)q^{m/d}.
\]
\end{theorem}

\begin{proof}
Suppose first that the characteristic of $\F_q$ is at least five.
The polynomial $x^3+ax^2+bx$
is equivalent to $x^3$ if $a^2=3b$,
and to $x^3-\tau x$ otherwise,
where $\tau$ is a nonzero element of $\F_q$.
Consequently,
\[
\sum_{a,b\in\F_q}|\mathcal{I}(x^3+ax^2+bx,1,m,q)|
=
q|\mathcal{I}(x^3,1,m,q)|
+(q^2-q)|\mathcal{I}(x^3-\tau x,1,m,q)|.
\]
According to Equation~\eqref{eq:TSRI}, using Theorem~\ref{thm:x^3} in its formulation of Equation~\eqref{eq:x^3},
and Theorem~\ref{thm:three_ram_points}, we find that
the above sum equals
\[
\frac{q}{3m_3}
\bigl(qI(m/m_3,q^{m_3})+I(m/m_3,\varepsilon q^{m_3})-qI(m/m_3,\varepsilon)-I(m/m_3,1)\bigr).
\]
This can be expanded as
\[
\frac{q}{3m}\sum_{d\mid n,\ 3\nmid d}
\mu(d)\bigl(q+\varepsilon^{m/d}\bigr)\bigl(q^{m/d}-\varepsilon^{m/d}\bigr),
\]
and leads to the desired conclusion.

Now suppose $\F_q$ has characteristic two.
Then the polynomial $x^3+ax^2+bx$ is equivalent to $x^3$ if $a^2=b$,
and to $x^3+x^2$ otherwise.
It suffices now to replace
$\mathcal{I}(x^3-\tau x,1,m,q)$
with
$\mathcal{I}(x^3-x^2,1,m,q)$
in the above argument, invoke Theorem~\ref{thm:char_two_(3,2)}
instead of Theorem~\ref{thm:three_ram_points},
and the desired conclusion follows.

Finally, suppose $\F_q$ has characteristic three.
If if $a=b=0$ then the polynomial $x^3+ax^2+bx$
equals $x^3$, but because $f(x^3)$ cannot be irreducible this case gives no contribution to our count.
If $a\neq 0$, then $x^3+ax^2+bx$ is seen to be equivalent to $x^3+x^2$
(by a linear substitution in $x$ followed by subtraction of a constant).
If $a=0$ and $b\neq 0$, then $x^3+ax^2+bx$ is equivalent to either $x^3-x$ or $x^3-\sigma x$
according as to whether $-b$ is a square in $\F_q$, or not.
Here $\sigma$ is a nonsquare in $\F_q$ as in Theorem~\ref{thm:char_three_(3)}.
Consequently,
\[
\begin{aligned}
\sum_{a,b\in\F_q}|\mathcal{I}(x^3+ax^2+bx,1,m,q)|
&=
(q^2-q)|\mathcal{I}(x^3+x^2,1,m,q)|
\\&\quad
+\frac{q-1}{2}|\mathcal{I}(x^3-x,1,m,q)|
\\&\quad
+\frac{q-1}{2}|\mathcal{I}(x^3-\sigma x,1,m,q)|.
\end{aligned}
\]
According to Equation~\eqref{eq:TSRI}, and using Theorems~\ref{thm:char_three_(3,2)}
and~\ref{thm:char_three_(3)},
we find that the above sum equals
\[
\frac{q-1}{3m_3}
\biggl((q+1)I(m/m_3,q^{m_3})+\frac{1}{2}I(m/2m_3,q^{2m_3})\biggr).
\]
This leads to the desired conclusion after expanding the sums.
\end{proof}

As mentioned in~\cite{Jiang-Yang}, the results of Theorem~\ref{thm:TSRI}
can be expressed in terms of counting polynomials of various types if one so wishes,
akin to what we have done in previous sections.
The most pleasant expression is obtained when $q\equiv 1\pmod{3}$, where we find
\[
|\TSRI(m,3,q)|=
\frac{|\GL_m(\F_q)|}{q^m-1}
\cdot
\frac{q(q+1)}{3m_3}\bigl(I(m/m_3,q^{m_3})-I(m/m_3,1)\bigr).
\]

\bibliography{References}
\end{document}